\documentclass[reqno,12pt]{amsart}

\usepackage{color}
\usepackage{amsmath}
\usepackage{amsfonts}
\usepackage{amssymb}
\usepackage{eufrak}
\usepackage{amscd}
\usepackage{amsthm}
\usepackage{amstext}
\usepackage[all]{xy}

\newcommand{\Ker}{\operatorname{Ker}}

\newcommand{\supp}{\operatorname{supp}}

   \theoremstyle{plain}
   \newtheorem{thm}{Theorem}[section]
   \newtheorem{prop}[thm]{Proposition}
   \newtheorem{lem}[thm]{Lemma}
   
   \theoremstyle{definition}

   \newtheorem{example}[thm]{Example}
   \theoremstyle{remark}
   
   \newtheorem{remark}[thm]{Remark}

\author{V. Manuilov, Jingming Zhu}

\date{}

\address{Moscow State University, Leninskie Gory 1, Moscow, 119991, Russia}

\email{manuilov@mech.math.msu.su}

\address{Jiaxing University, 56 South Yuexiu Road, Jiaxing, Zhejiang, 314001, China}

\email{jingmingzhu@mail.zjxu.edu.cn}

\title[Uniform Roe algebras without bounded geometry]{Two versions of the uniform Roe algebras for spaces without bounded geometry}

\begin{document}

\maketitle

\begin{abstract}
We consider two versions of the uniform Roe algebra for uniformly discrete spaces without bounded geometry and discuss some of their properties.
\end{abstract}

\section*{Introduction}

Various versions of Roe algebras are useful in study of both discrete metric spaces and $C^*$-algebras. Although the uniform Roe algebra is well defined for any discrete metric space $X$, most known results are obtained in the case of metric spaces of bounded geometry, when for any $R>0$ there is a uniform bound for the number of points in all balls of radius $R$. Without the condition of bounded geometry, the uniform Roe algebra seems too great, so our idea is to make it smaller in two different ways. The first way is to consider the direct limit of the uniform Roe algebras of all subspaces of bounded geometry. In this way we get the point-related uniform Roe algebra $C^*_{u,p}(X)$, for which some properties of the uniform Roe algebras of bounded geometry spaces can be extended, e.g. the property A of Guoliang Yu implies nuclearity of $C^*_{u,p}(X)$. But this algebra seems to be too small, so we consider also another version, obtained by varying the metric on $X$. We consider only metrics of bounded geometry that dominate the given metric, and show that we can pass to the direct limit with respect to these metrics, obtaining the metric-related uniform Roe algebra $C^*_{u,m}(X)$. For the latter, we show that its maximal version is isomorphic to its regular version if $X$ satisfies the Higson--Roe condition and is coarsely of bounded geometry, which generalizes a similar result of Spakula and Willett for spaces of bounded geometry \cite{Spakula-Willett}.

\section{Two versions of the uniform Roe algebras for spaces without bounded geometry}

Let $X$ be a discrete metric space with a metric $d_0$. It is called {\it uniformly discrete} if $\inf\{d_0(x,y):x,y\in X,x\neq y\}>0$, and it {\it has bounded geometry} if for every $R>0$ the balls $B_{d_0}(x,R)$ of radius $R$ centered at $x\in X$ contain a uniformly bounded number of points, i.e. $\sup_{x\in X}|B_{d_0}(x,R)|<\infty$, where $|A|$ denotes the number of points in the set $A$.

For convenience only, we allow metrics to take the value $\infty$.

Let $l^2(X)$ denote the Hilbert space of square-summable functions on $X$, with the standard orthonormal basis $(\delta_x)_{x\in X}$. An operator $T$ on $l^2(X)$ is of {\it propagation} less than $S$ if $d_0(x,y)>S$ implies that the matrix entries $T_{xy}=\langle\delta_x,T\delta_y\rangle$ equal 0. Let $C^S(X)=C^S(X,d_0)$ denote the set of all operators of propagation less than $S$, and let $\mathbb C_u[X]=\mathbb C_u[X,d_0]=\cup_{S\in\mathbb R}C^S(X)$ be the $*$-algebra of all operators of finite propagation. The standard definition of the uniform Roe algebra $C^*_u(X)$ of $(X,d_0)$ is the norm closure of the finite propagation operators, $C^*_u(X)=C^*_u(X,d_0)=\overline{\mathbb C_u[X,d_0]}$.

If $d_0$ is a metric of bounded geometry on $X$ and if $T$ is a finite propagation operator on $X$ then it is easy to see that
$\sup_{x\in X}|\{y\in X:T_{xy}\neq 0\}|<\infty$, which makes certain arguments work. Without bounded geometry condition the uniform Roe algebra may be too great. For example, if $X$ is countable with the metric $d_0$ defined by $d_0(x,y)=1$ whenever $x\neq y$ then $C^*_u(X,d_0)=\mathbb B(l^2(X))$, i.e. all bounded operators on the Hilbert space $l^2(X)$.

There are at least two possibilities to make this algebra smaller. One way is to remove `extra' {\it points}, so that remaining points would be a bounded geometry space. So consider all subspaces $Y\subset X$ such that $d_0|_Y$ is a metric of bounded geometry. Then $l^2(X)=l^2(Y)\oplus l^2(X\setminus Y)$, and there is a canonical inclusion $\mathbb C_u[Y]\subset \mathbb C_u[X]$, which extends to a canonical inclusion $C^*_u(Y)\subset C^*_u(X)$, so we can set $\mathbb C_{u,p}[X]=\operatorname{dirlim}_{Y}\mathbb C_u[Y]$.
Let $C^*_{u,p}(X)$ be the norm closure, in $\mathbb B(l^2(X))$, of $\mathbb C_{u,p}[X]$. Note that $C^*_{u,p}(X)=\operatorname{dir\ lim}_{Y}C^*_u(Y)$. 
We call this algebra {\it the point-related uniform Roe algebra}. When $d_0(x,y)=1$ for any $x\neq y$, we get $C^*_{u,p}(X)=\mathbb K(l^2(X))$ --- the algebra of compact operators. For some purposes, this is a good definition, for example, these $C^*$-algebras are Morita equivalent for coarsely equivalent spaces by Proposition \ref{Morita}, but usually the uniform Roe algebra should be greater.

The second possibility is to change {\it metric}, not the points, so that with respect to the new metric our space is of bounded geometry. For this purpose, new metrics should increase distance between points. We are going to use metrics of bounded geometry on $X$, and to explain how to order them and to pass to the limit.

For a metric space $(X,d_0)$, let $\mathcal D=\mathcal D(d_0)$ consists of all metrics $d$ on $X$ satisfying the following properties:
\begin{itemize}
\item[(D1)]
$\inf\{d(x,y):x,y\in X,x\neq y\}\geq 1$;
\item[(D2)]
there exists $C>0$ such that $d(x,y)\geq \frac{1}{C}d_0(x,y)$ for any $x,y\in X$ (if $d_0(x,y)=\infty$ then $d(x,y)$ should also be infinite);
\item[(D3)]
$d$ is of bounded geometry.
\end{itemize}

Let us show that $\mathcal D$ can be made a directed set.

For $d_1,d_2\in\mathcal D$ we say that $d_1\preceq d_2$ if $d_2(x,y)\leq d_1(x,y)$ for any $x,y\in X$.

\begin{prop}
If the metric $d_0$ is uniformly discrete then the set $\mathcal D=\mathcal D(d_0)$ is directed, i.e. for any $d_1,d_2\in\mathcal D$ there exists $d\in\mathcal D$ such that $d_1\preceq d$, $d_2\preceq d$.

\end{prop}
\begin{proof}
By assumption, we may assume that $\inf_{x\neq y}d_0(x,y)\geq 1$.

Let us construct a graph with $X$ as the set of vertices. For any $x,y\in X$, $x\neq y$, connect these two vertices by an edge of length $l_{x,y}=\min(d_1(x,y),d_2(x,y))$, and define a metric $d$ on $X$ as the geodesic metric, i.e. $d(x,y)=\inf \sum_{i=1}^r l_{x_{i-1},x_i}$, where the infimum is taken over all $r\in\mathbb N$ and all sets $x_0=x,x_1,\ldots,x_{r-1},x_r=y$ of points in $X$. Then $d(x,y)\leq l_{x,y}$ for any $x,y\in X$, hence $d_1,d_2\preceq d$. It is also clear that $d$ satisfies (D1).

Let us check (D2). Let $C$ satisfy $d_j(x,y)\geq\frac{1}{C}d_0(x,y)$ for any $x,y\in X$ and for $j=1,2$.
Fix $x,y\in X$ and some $\varepsilon>0$, and let $x_1,\ldots,x_{r-1}\in X$ be points such that
\begin{equation}\label{1}
d(x,y)\geq \sum_{i=1}^r l_{x_{i-1},x_i}-\varepsilon.
\end{equation}
Then
\begin{equation}\label{2}
\sum_{i=1}^r l_{x_{i-1},x_i}= \sum_{i=1}^r d_{j_i}(x_{i-1},x_i)\geq
\frac{1}{C}\sum_{i=1}^r d_0(x_{i-1},x_i)\geq \frac{1}{C}d_0(x,y),
\end{equation}
where $j_i$ is either 1 or 2.
It follows from (\ref{1}) and (\ref{2}) that $d(x,y)\geq\frac{1}{C}d_0(x,y)-\varepsilon$ for any $\varepsilon>0$, hence the metric $d$ satisfies (D2).

Fix $x\in X$ and radius $R$. If $y\in B_d(x,R)$, i.e. if $d(x,y)\leq R$, then, by the definition of the metric $d$, there exists a set $x=x_0,x_1,\ldots,x_{r-1},x_r=y$ of points such that
\begin{itemize}
\item[(i)]
$\sum_{i=1}^r d(x_{i-1},x_i)\leq R+1$;
\item[(ii)]
for each $i=1,\ldots,r$, $d(x_{i-1},x_i)=l_{x_{i-1},x_i}$.
\end{itemize}

We claim that $r\leq R+1$. Indeed, since $l_{x,z}\geq 1$ for any $x,z\in X$,
$$
R+1\geq\sum_{i=1}^r l_{x_{i-1},x_i}\geq r.
$$

Now, instead of estimating the number of points $y\in B_d(x,R)$ we estimate the number of points $x_1,\ldots,x_r$ with the above properties (i) and (ii).
As both $d_1$ and $d_2$ are metrics of bounded geometry, for a given $R$, there exists $M>0$ such that $|B_{d_1}(x,R)|\leq M$ and $|B_{d_2}(x,R)|\leq M$ for any $x\in X$.

We know that $x_1\in B_d(x_0,R)$ and that $d(x_0,x_1)$ equals either $d_1(x_0,x_1)$ or $d_2(x_0,x_1)$, hence the number of all possible $x_1$ satisfying (i) and (ii) is bounded by $2M$. Similarly, $x_2\in B_d(x_1,R)$, and we have not more than $2M$ possibilities for $x_2$ for a given $x_1$, and hence not more than $(2M)^2$ possibilities for $x_2$ totally. Proceeding inductively, there are not more than $(2M)^r\leq (2M)^{R+1}$ possibilities for $x_r=y$. This proves the bounded geometry condition (D3) for $d$.

\end{proof}

For $d_1,d_2\in \mathcal{D}$ with $d_1\preceq d_2$, and for any $S\in(0,\infty)$, there is a canonical inclusion $C^S(X,d_1)\to C^S(X,d_2)$. Indeed, let $T\in C^S(X,d_1)$. If $d_2(x,y)\geq S$ then $d_1(x,y)\geq d_2(x,y)\geq S$, hence $T_{xy}=0$.

So, now we can define the $*$-algebra $\mathbb C_{u,m}[X,d]=\operatorname{dirlim}_{\mathcal D}\mathbb C_u[X,d]$. Its closure in $\mathbb B(l^2(X))$ is called {\it the metric-related uniform Roe algebra} and is denoted by $C^*_{u,m}(X,d_0)=C^*_{u,m}(X)$, where the subscript $m$ indicates that we are changing {\it metrics} on $X$.

\medskip
Given a metric space $(X, d_0)$, let $Y\subset X$ be a subspace of bounded geometry with respect to the metric $d_0|_{Y}$. Define a metric $d$ on $X$ by setting $d|_Y=d_0|_Y$, and $d(x,y)=d(x,z)=\infty$ for any $x,z\notin Y$, $y\in Y$. Then $d\in\mathcal D$, hence there is a canonical inclusion $i_Y:\mathbb C_u[Y,d_0|_Y]\to \mathbb C_u[X,d]$. If $Y'\supset Y$ is another subset of bounded geometry in $X$ then the metric $d'\in\mathcal D$ satisfies $d\preceq d'$, and the diagram
$$
\begin{xymatrix}{
\mathbb C_u[Y,d_0|_Y]\ar[r]^-{i_Y}\ar[d]&\mathbb C_u[X,d]\ar[d]\\
\mathbb C_u[Y',d_0|_{Y'}]\ar[r]^-{i_{Y'}}&\mathbb C_u[X,d']
}\end{xymatrix}
$$
commutes, so we can pass to the limit to obtain the map $i:\mathbb C_{u,p}[X,d_0]\to\mathbb C_{u,m}[X,d_0]$, which is obviously injective.

\section{Another description of $\operatorname{dirlim}_{\mathcal D}\mathbb C_u[X,d]$}

Fix $k\in\mathbb N$, and denote by $\mathbb B^{(k)}(l^2(X))$ the set of all $T\in\mathbb B(l^2(X))$ that satisfy the following property: each line and each column of the matrix of $T$ with respect to the standard basis $\{\delta_x\}_{x\in X}$ of $l^2(X)$ contains no more than $k$ non-zero entries.

\begin{lem}\label{lem-graph}
Let $T\in\mathbb B^{(k)}(l^2(X))\cap C^{S}(X,d_0)$. Then there exists a metric $d\in\mathcal D$ such that $T$ has propagation 1 with respect to $d$.

\end{lem}
\begin{proof}
Let $\Gamma$ be the graph with $X$ as the set of vertices, and let $x,y\in X$ be connected by an edge of length 1 if $T_{xy}\neq 0$ or $T_{yx}\neq 0$. Let $d$ be the geodesic metric on $\Gamma$. The property (D1) for $d$ is trivially true. Moreover, the metric $d$ is integer-valued, $d(x,y)\in\mathbb N\cup\{\infty\}$ for any $x,y\in X$.

To check (D2), let us first assume that $d(x,y)=1$. This implies that at least one of $T_{xy}$ and $T_{yx}$ is non-zero, hence $d_0(x,y)\leq S$. Then $d(x,y)\geq \frac{1}{S}d_0(x,y)$. If $x$ and $y$ are arbitrary with $d(x,y)=r\in\mathbb N$ then there is a set $x=x_0,x_1,\ldots,x_{r-1},x_r=y$ such that $d(x_{i-1},x_i)=1$ for any $i=1,2,\ldots,r$. Then
$$
d(x,y)=\sum_{i=1}^rd_{x_{i-1},x_i}\geq\frac{1}{S}\sum_{i=1}^rd_0(x_{i-1},x_i)\geq\frac{1}{S}d_0(x,y).
$$

Property (D3) follows from $T\in\mathbb B^{(k)}(l^2(X))$. Indeed, the unit ball $B_d(x,1)$ contains not more than $2k+1$ points for any $x\in X$, and the size of any ball can be estimated by sizes of unit balls. Thus, $d\in\mathcal D$.

Finally, it is clear that the propagation of $T$ with respect to $d$ is 1.

\end{proof}

\begin{prop}
Let $(X,d_0)$ be a uniformly discrete metric space, and let $\mathcal{D}$ be the directed set defined by (D1)-(D3). Then $\operatorname{dirlim}_\mathcal D \mathbb C[X,d]=\cup_{k,S}\mathbb B^{(k)}(l^2(X))\cap C^S(X,d_0)$.

\end{prop}
\begin{proof}
If $T\in \operatorname{dirlim}_\mathcal D \mathbb C[X,d]$ then there exists $d\in\mathcal D$ such that $T\in\mathbb C[X,d]$, and the latter means that there exists $S>0$ such that $T\in C^S(X,d)$. By property (D2), there exists $C>0$ such that $d(x,y)\geq\frac{1}{C}d_0(x,y)$ for any $x,y\in X$. Then $T\in C^{CS}(X,d_0)$. Since the metric $d$ is of bounded geometry, there exists $k\in\mathbb N$ such that $|\{y\in X:d(x,y)\leq S\}|\leq k$ for any $x\in X$, hence the number of non-zero entries in each line and in each column of the matrix of $T$ is bounded by $k$, i.e. $T\in\mathbb B^{(k)}(l^2(X))$.

If $T\in \mathbb B^{(k)}(l^2(X))\cap C^S(X,d_0)$ then, by Lemma \ref{lem-graph} there exists a metric $d\in\mathcal D$ such that $T\in C^1(X,d)$.

\end{proof}

\begin{remark}
Note that we could define one more version of the Roe algebra, by changing both the points and the metric. Namely, there is a directed set $\mathcal E$ consisiting of all pairs $(Y,d)$, where $Y\subset X$, $d$ is a metric on $X$ such that it satisfies (D1) and (D2), while (D3) is replaced by
\begin{enumerate}
\item[(D3')]
$Y$ is of bounded geometry with respect to $d$.
\end{enumerate}
Then $(Y,d)\preceq(Y',d')$ if $Y\subset Y'$ and $d\preceq d'$.
It turns out that the direct limit over $\mathcal E$ is the same as that over $\mathcal D$.
\end{remark}

\section{Coarse equivalence vs Morita equivalence}

\begin{prop}\label{Morita}
Let $X=(X,d_X)$, $Y=(Y,d_Y)$ be coarse equivalent uniformly discrete spaces with uniformly discrete metrics $d_X$ and $d_Y$ respectively. Then the algebras $C^*_{u,p}(X)$ and $C^*_{u,p}(Y)$ are Morita equivalent.

\end{prop}
\begin{proof}
The proof is a slight modification of the proof of Theorem 4 in \cite{BNW}.

The general case standardly reduces to the case, when the coarse equivalence map $f:X\to Y$ is surjective. Recall that a coarse map $h:X\to Y$ should satisfy
\begin{equation}\label{coarse}
\forall R>0 \
\exists S>0 \mbox{\ such that\ }d_X(x,x')\leq R\mbox{\ implies\ }d_Y(f(x),f(x'))\leq S.
\end{equation}
As $X$ and $Y$ are coarse equivalent, there exists a coarse map $g:Y\to X$ and some $C>0$ such that $d_X(g\circ f(x),x)\leq C$ for any $x\in X$.

We claim that
$f(A)$ is of bounded geometry for any $A\subset X$ of bounded geometry.
Indeed, let $A\subset X$ be of bounded geometry, $y=f(x)\in f(A)$.
Surjectivity of $f$ implies that
$$
|B_{d_Y}(f(x),R)|\leq|\{x'\in A:d_Y(f(x),f(x'))\leq R\}|.
$$
If $d_Y(f(x),f(x'))\leq R$ then, by (\ref{coarse}) applied to $g$, we have
$$
d_X(x,x')\leq d_X(x,g\circ f(x))+d_X(g\circ f(x),g\circ f(x'))+d_X(g\circ f(x'),x')\leq S+2C,
$$
hence $|B_{d_Y}(f(x),R)|\leq|B_{d_X}(x,S+2C)\cap A|$, and the latter is bounded uniformly in $x$.

Note that $f|_A$ is a coarse equivalence map implementing the coarse equivalence between $A$ and $f(A)$.
By Theorem 4 in \cite{BNW}, bounded geometry of $A$ implies existence of an isomorphism
$$
\alpha_A:C^*_u(A)\otimes\mathbb K\to C^*_u(f(A))\otimes\mathbb K,
$$
induced by a bijection
$$
\phi_A:A\times\mathbb N\to f(A)\times\mathbb N
$$
given by the formula
$$
\phi_A(x,j)=(f(x),\pi(x)+jN(f(x))),
$$
where $N(y)$, for $y\in f(A)$, denotes the (finite) cardinality of $f^{-1}(y)\cap A$,  and $\pi:f^{-1}(y)\mapsto\{1,2,\ldots,N(f(x))\}$ is a bijection. If $A\subset A'\subset X$ and if $A'$ is of bounded geometry then $\phi_{A'}|_{A\times\mathbb N}=\phi_A$, hence the diagram
$$
\begin{xymatrix}{
C^*_u(A)\otimes\mathbb K\ar[r]^-{\alpha_A}\ar[d]&C^*_u(f(A))\otimes\mathbb K\ar[d]\\
C^*_u(A')\otimes\mathbb K\ar[r]^-{\alpha_{A'}}&C^*_u(f(A'))\otimes\mathbb K
}\end{xymatrix}
$$
commutes.
Hence we can define the direct limit map $\alpha:C^*_{u,p}(X)\otimes\mathbb K\to C^*_{u,p}(Y)\otimes\mathbb K$.

Note that for each $B\subset Y$ of bounded geometry there exists $A\subset X$ of bounded geometry such that $B=f(A)$. Indeed, for each $y\in B$ choose one point $x\in X$ such that $f(x)=y$, and let $A$ be the set of all such points. Then $f|_A$ is a bijection between $A$ and $B=f(A)$. For $x\in A$, by (\ref{coarse}) and bijectivity of $f|_A$, we have $|B_{d_X}(x,R)|\leq |B_{d_Y}(f(x),S)|$ for any $x\in A$, which proves bounded geometry for $A$.

As each $\alpha_A$ is an isomorphism, this shows that the limit map $\alpha$ is an isomorphism.

\end{proof}

\begin{remark}
This result doesn't hold for the metric-related version of the uniform Roe algebra. The space $\mathbb N$ with the metric $d_0$ given by $d_0(x,y)=1$ whenever $x\neq y$ is coarsely equivalent to a single point, but the metric-related uniform Roe algebras are $\mathbb B_f(H)$ from \cite{Manuilov1} and $\mathbb C$ respectively.

\end{remark}

\section{Higson--Roe condition and nuclearity}

Recall that a metric space $X$ satisfies the Higson--Roe (HR) condition \cite{Higson-Roe} if for any $R>0$ and any $\varepsilon>0$ there exists $S>0$ and $\xi:X\to l^2(X)$ such that
\begin{itemize}
\item[(HR1)]
the function $\xi_x$ takes values in $[0,1]$ and $\|\xi_x\|=1$ for each $x\in X$;
\item[(HR2)]
$\|\xi_x-\xi_y\|<\varepsilon$ when $d(x,y)\leq R$;
\item[(HR3)]
$\supp\xi_x\subset B_d(x,S)$ for any $x\in X$.
\end{itemize}

It is known that the property A of Guoliang Yu implies the HR condition, and that these two properties are equivalent if $X$ is of bounded geometry (\cite{Willett}, Theorem 1.2.4).

\begin{prop}
Let $(X,d_0)$ satisfy the HR condition. Then $C^*_{u,p}(X,d_0)$ is nuclear.

\end{prop}
\begin{proof}
If $Y\subset X$ is of bounded geometry then, by Proposition 3 of \cite{Dranishnikov}, $Y$ satisfies the HR condition, hence, by Theorem 5.5.7 of \cite{Brown-Ozawa}, $C^*_u(Y)$ is nuclear. Direct limit of nuclear $C^*$-algebras is nuclear as well \cite{Takesaki}.

\end{proof}

\begin{prop}\label{N}
Suppose $(X,d_0)$ is not of bounded geometry. Then $C^*_{u,m}(X,d_0)$ is not nuclear.

\end{prop}
\begin{proof}
First, we need the following statement:

\begin{lem}
Let $X$ be not of bounded geometry. Then there exists $R>0$ and a sequence $(X_n)_{n\in\mathbb N}$ of finite subspaces such that
\begin{itemize}
\item[(x1)]
$X_n\cap X_m=\emptyset$ when $n\neq m$;
\item[(x2)]
$\operatorname{diam}X_n\leq R$;
\item[(x3)]
$\lim_{n\to\infty}|X_n|=\infty$.
\end{itemize}

\end{lem}
\begin{proof}
Consider the two possible cases:

1) There exists $R>0$ and $x\in X$ such that $B_d(x,R)$ is infinite. Then it is trivial to find the sets $X_n$, $n\in\mathbb N$, inside $B_d(x,R)$.

2) For any $R>0$, all balls $B_d(x,R)$ are finite. As $X$ does not have bounded geometry, there exists $R>0$ and a sequence $\{x_n\}_{n\in\mathbb N}$ of points such that $\lim_{n\to\infty}|B_d(x_n,R)|=\infty$. Let us pass to a subsequence. Set $n_1=1$. As $B_d(x_1,2R)$ contains only a finite number of points, there exists $n_2$ such that $|B_d(x_{n_2},R)|>2$ and $x_{n_2}\notin B_d(x_{n_1},2R)$. The latter means that $B_d(x_{n_1},R)\cap B_d(x_{n_2},R)=\emptyset$. Inductively, after we choose $n_1,\ldots, n_k$ such that $B_d(x_{n_j},R)\cap B_d(x_{n_j},R)=\emptyset$ for $i\neq j$ and $|B_d(x_{n_i},R)|>i$, we can find $n_{k+1}$ such that
$x_{n_{k+1}}\in X\setminus(\sqcup_{i=1}^k B_d(x_{n_i},2R))$ and $|B_d(x_{n_{k+1}},R)|>k+1$.

\end{proof}

To proceed with the proof of Proposition \ref{N}, let $\Gamma$ be a countable residually finite property (T) group, $\Gamma_1\supset\Gamma_2\supset\cdots$ a sequence of finite index normal subgroups with trivial intersection. Let $\{\pi_n\}_{n\in\mathbb N}$ be a set of mutually non-equivalent irreducible finitedimensional representations of $\Gamma$ such that $\pi_n$ factorises through $\Gamma/\Gamma_n$ for each $n\in\mathbb N$. Let also $\lambda_n$ denote the quasiregular representation of $\Gamma$, i.e. the composition of the regular representation of $\Gamma/\Gamma_n$ with the quotient map $\Gamma\to\Gamma/\Gamma_n$.

Let $C^*_{\oplus_{n\in\mathbb N}\pi_n}(\Gamma)$ (resp. $C^*_{\oplus_{n\in\mathbb N}\lambda_n}(\Gamma)$) denote the $C^*$-algebra generated by all $\oplus_{n\in\mathbb N}\pi_n(g)$ (resp. $\oplus_{n\in\mathbb N}\lambda_n(g)$), $g\in\Gamma$. As $\lambda_n$ contains $\pi_n$ for any $n\in\mathbb N$, $C^*_{\oplus_{n\in\mathbb N}\pi_n}(\Gamma)\subset C^*_{\oplus_{n\in\mathbb N}\lambda_n}(\Gamma)$.

It is known \cite{Wassermann} that $C^*_{\oplus_{n\in\mathbb N}\pi_n}(\Gamma)$ is not exact. So, if we show that $C^*_{\oplus_{n\in\mathbb N}\lambda_n}(\Gamma)$ is a $C^*$-subalgebra in $C^*_{u,m}(X,d_0)$ then the latter cannot be nuclear.

Let $|\Gamma/\Gamma_n|=m_n$. Passing to a subsequence, we may assume that there exists $R>0$ and a sequence $\{X_n\}_{n\in\mathbb N}$ of subsets in $X$ such that $X_n\cap X_m=\emptyset$ when $n\neq m$ and $|X_n|=m_n$. Then fix, for each $n\in\mathbb N$, a bijection $\varphi_n:\Gamma/\Gamma_n\to X_n$.

Let $p_n:l^2(X)\to l^2(X_n)$ and $i_n:l^2(X_n)\to l^2(X)$ be canonical projection and canonical inclusion induced by the inclusion $X_n\subset X$.
For $\xi\in l^2(X)$ and $g\in\Gamma$, set $\varphi(g)\xi=\oplus_{n\in\mathbb N}i_n\lambda_n(g) p_n\xi$. The map $\varphi$ extends to an inclusion $\varphi:C^*_{\oplus_{n\in\mathbb N}\lambda_n}(\Gamma)\to \mathbb B(l^2(X))$. By (x2), the propagation of $\varphi(g)$ does not exceed $R$, hence $\varphi(g)\in C^*_u(X,d)$ for any $g\in\Gamma$. As $\varphi(g)$ acts on each $X_n$ by permutations, its matrix has exactly one non-zero entry in each line and in each column. As linear combinations of group elements is dense in $C^*_{\oplus_{n\in\mathbb N}\lambda_n}(\Gamma)$, the image of $\varphi$ lies in $C^*_{u,m}(X,d_0)$.

\end{proof}

\section{HR condition and the maximal uniform Roe algebras}

Let us consider the $*$-algebras $\mathbb C_{u,p}[X,d]$ and $\mathbb C_{u,m}[X,d]$ before completing them with respect to the norm in $l^2(X)$. Following \cite{Spakula-Willett}, we can define the maximal $C^*$-algebras $C^*_{u,p;max}(X)$ and $C^*_{u,m;max}(X)$ by completing these $*$-algebras with respect to the supremum of norms $\|\pi(\cdot)\|$ over all $*$-representations, if the latter is finite.

\begin{lem}
Let $T\in \mathbb C_{u,p}[X,d]$ or $T\in \mathbb C_{u,m}[X,d]$. Then $\sup\|\pi(T)\|<\infty$.

\end{lem}
\begin{proof}
If $T\in\mathbb C_{u,p}[X,d]$ then there exists $Y\subset X$ of bounded geometry and $S\in\mathbb N$ such that $T\in C^S(Y)$. Then $T=\sum_{i=1}^ST^{(i)}$, where  each $T^{(i)}$ has the form $T^{(i)}=f_iv_i$, where $f_i\in l^\infty(Y)$, $\|f_i\|_\infty\leq\|T\|$, and $v_i$ are partial isometries.
Similarly, if $T\in\mathbb B^{(S)}(l^2(X))$ then $T$ is of the same form.
Therefore,
$$
\|\pi(T)\|\leq\sum_{i=1}^S\|\pi(f_i)\|\|\pi(v_i)\|\leq S\|T\|.
$$

\end{proof}

There is a canonical surjection $\lambda:C^*_{u,m;max}(X)\to C^*_{u,m}(X)$, and we are interested in conditions providing that $\lambda$ is an isomorphism.

Recall that $(X,d_0)$ is coarsely of bounded geometry \cite{Capraro} if there exists $l>0$ and a subspace $X_l\subset X$ of bounded geometry, which is an $l$-net in $X$.

\begin{lem}
Let $X_l\subset X$ be an $l$-net of bounded geometry, and let $X$ satisfy the HR condition. Then the functions $(\xi_x)_{x\in X}$ in the HR condition can be taken to satisfy additionally $\supp\xi_x\in X_l$.

\end{lem}
\begin{proof}
By Proposition 3 of \cite{Dranishnikov}, the HR condition passes to subspaces, so for any $\varepsilon>0$ and any $R>2l$ there exists $S>0$ and $\xi:X_l\to l^2(X_l)$ with the aforementioned properties. Fix, for any $x\in X$, a point $p(x)\in X_l$ such that $d_0(x,p(x))\leq l$. This gives a coarse equivalence map $p:X\to X_l$. Set
$$
\eta_x(y)=\left\lbrace\begin{array}{cl}\xi_{p(x)}(y), &\mbox{if\ }y\in X_l;\\0,&\mbox{otherwise.}\end{array}\right.
$$
Then $\supp\eta_x\in X_l$ for any $x\in X$, and (HR1) obviously holds true. If $d_0(x,y)<R-2l$ then $d_0(p(x),p(y))<R$, hence
$$
\|\eta_x-\eta_y\|_{l^2(X)}=\|\eta_x|_{X_l}-\eta_y|_{X_l}\|_{l^2(X_l)}=\|\xi_{p(x)}-\xi_{p(y)}\|_{l^2(X_l)}<\varepsilon,
$$
which implies (HR2). Finally, since $\supp\eta_x=\supp\xi_{p(x)}\subset B_{d_0}(x,S+l)$,  (HR3) holds.

\end{proof}

\begin{thm}\label{lambda-iso}
Let $(X,d_0)$ be uniformly discrete, coarsely of bounded geometry and satisfy the HR condition. Then $C^*_{u,m;max}(X)=C^*_{u,m}(X)$.

\end{thm}
\begin{proof}
We follow the proof of Proposition 1.3 of \cite{Spakula-Willett} with only a few changes caused by unbounded geometry.
Let $\mathbb C_u[X,d_0]\subset A$ for some $C^*$-algebra $A$. We are going to construct a family of maps $M_k:\mathbb C_u[X,d_0]\to \mathbb C_u[X,d_0]$, which are completely positive with respect to the order structure in $A$.

Given $\varepsilon>0$ and $R>0$, let $(\xi_x)_{x\in X}$ be the set of functions provided by the HR condition and coarsely bounded geometry. Then we can assume that $\supp\xi_x\in X_l$ for any $x\in X$. Set $k(x,y)=\langle\xi_x,\xi_y\rangle$, $M_k(T)=k\ast T$, where $\ast$ denotes the Schur product, i.e. element-wise matrix product defined by $[M_k(T)]_{xy}=k(x,y)T_{xy}$, $x,y\in X$. Note that if $T\in C^S(X,d)$ then $M_k(T)\in C^S(X,d)$.

By Lemma 2.2 of \cite{Spakula-Willett}, for any $T\in\mathbb C_u[X,d]$, bounded geometry of $X_l$ implies that there exists $N\in\mathbb N$ and functions $\varphi_i:X\to[0,1]$, $i=1,2,\ldots,N$, such that $M_k(T)=\sum_{i=1}^N\varphi_iT\varphi_i$. This shows complete positivity of $M_k$ in any $C^*$-completion of $\operatorname{dirlim}_\mathcal D\mathbb C_u[X,d]$. In particular, this implies that $M_k$ is bounded. By continuity, the map $M_k$ can be extended to $*$-homomorphisms
$$
M^{max}_k:C^*_{u,m;max}(X,d_0)\to C^*_{u,m;max}(X,d_0)
$$
and
$$
M^{\lambda}_k:C^*_{u,m}(X,d_0)\to C^*_{u,m}(X,d_0).
$$
It is clear that $C^S(X,d)$ is closed in $C^*_{u,m;max}(X,d_0)$ for any $S>0$ and any $d\in\mathcal D$ (cf. Lemma 2.4 of \cite{Spakula-Willett}). By varying $\varepsilon$ and $R$, we can find a sequence $k_n$ such that $M^{max}_{k_n}$ and $M^\lambda_{k_n}$ converge with respect to the point-norm topology to the identities on $C^*_{u,m;max}(X,d_0)$ and on $C^*_{u,m}(X,d_0)$ respectively. If $T\in\Ker\lambda$ then $\lambda(M^{max}_{k_n}(T))=M^\lambda_{k_n}(\lambda(T))=0$ for any $n\in\mathbb N$. But $M^{max}_{k_n}(T)$ lies in $\operatorname{dirlim}_{\mathcal D}\mathbb C[X,d]$, on which $\lambda$ is injective, so $M^{max}_{k_n}(T)=0$, and hence $T=0$ as their limit.

\end{proof}

Remark that a similar, but simpler, argument works for $C^*_{u,p,max}$ in the second case:
\begin{thm}
If $X$ is uniformly discrete and coarsely of bounded geometry then $\lambda:C^*_{u,p;max}\to C^*_{u,p}$ is an isomorphism.

\end{thm}

Here are two examples, when Theorem \ref{lambda-iso} is applicable.
\begin{example}
Let $X$ be countable, and $d_0(x,y)=1$ when $x\neq y$. Then $X$ is coarsely equivalent to a point. In this case, $C^*_{u,m}(X)$ is the $C^*$-algebra $\mathbb B_f(H)$ from \cite{Manuilov1}.

\end{example}

\begin{example}
Let $X=\sqcup_{n\in\mathbb N}X_n$, where $X_n$ consists of $n$ points at distance 1 between any two of them, and let $d_0(x,y)=\infty$ if $x\in X_n$, $y\in X_m$, $n\neq m$. Then $X$ is coarsely equivalent to $\mathbb N$ with infinite distance between points. In this case, $C^*_{u,m}(X)=\mathbb B_f(H)\cap\prod_{n\in\mathbb N}M_n$ is a $C^*$-subalgebra in the product of matrix algebras.

\end{example}

\end{document}